\documentclass[12pt,reqno]{amsart}
\usepackage{amsmath, amstext, amsbsy, amssymb}

\input xy \xyoption{all}

\numberwithin{equation}{section}
\def\beq{\begin{eqnarray}}
\def\eeq{\end{eqnarray}}
\def\beqs{\begin{eqnarray*}}
\def\eeqs{\end{eqnarray*}}

\def\mz{{\mathbb Z}}

\def\mc{{\mathbb C}}

\def\dim{{\hbox{\rm dim}}}

\newfont{\df}{eufm10}

\def\mc{{\mathbb C}}

\def\deg{\hbox{\rm deg}}

\def\sg{\frak{g}}
\def\sgt{\frak{g}_{\rm tor}}
\def\sga{\widehat{\frak{g}}}

\title[]
{Finite dimensional modules over quantum toroidal algebras}

\author[Xia]{Limeng Xia}

\thanks{E-mail address: xialimeng@ujs.edu.cn}

\date{}
\begin{document}
\maketitle
\centerline{\it Institute of Applied System Analysis, Jiangsu University}
\centerline{\it Zhenjiang  212013,  Jiangsu Prov. China}

\begin{abstract}
In this paper, for all generic $q\in\mc^*$, if $\sg$ is not of type $A_1$,  we prove that the quantum toroidal algebra $U_q(\sgt)$ has no nontrivial finite dimensional simple module.
 \end{abstract}

\newtheorem{theo}{Theorem}[section]
\newtheorem{theorem}[theo]{Theorem}
\newtheorem{defi}[theo]{Definition}
\newtheorem{lemma}[theo]{Lemma}
\newtheorem{coro}[theo]{Corollary}
\newtheorem{prop}[theo]{Proposition}
\newtheorem{remark}[theo]{Remark}
\vskip1cm

\section{Introduction}

Let $\sg$ be a finite dimensional complex simple Lie algebra and  $\widehat{\sg}=\sg\otimes\mc[t,t^{-1}]\oplus\mc c$ be the associated affine Lie algebra. Both quantum groups $U_q(\sg)$ and $U_q(\widehat{\sg})$ have finite dimensional simple modules which can be viewed as the quantization of simple modules over classic Lie algebras (\cite{CP1}, \cite{CP2}, \cite{CP3}).

In 1987, Drinfeld gave an extremely important realization of quantum affine algebras, then it was applied to construct the  affinization of quantum affine algebras. Such new algebras are called quantum toroidal algebras. Let $\sgt$ be the toroidal Lie algebra of $\sg$ with nullity $2$. The quantum toroidal algebra $U_q(\sgt)$ can be regarded as a quantum deformation of the enveloping algebra $U(\sgt)$ of $\sgt$. In the past decades, the  quantum toroidal algebras and their representations have been researched by many authors (see \cite{GKV}, \cite{FJW}, \cite{H}, \cite{M1}, \cite{M2}, \cite{M3}, \cite{Sa},  \cite{VV}).

Certainly, $\sgt$ has nontrivial finite dimensional modules. However,  no one has constructed some finite dimensional modules for $U_q(\sgt)$ while $q$ is generic. In this paper, we prove the following result.
\begin{theorem}
If $\sg$ is not of type $A_1$ and $q\in\mc^*$ is generic, then quantum toroidal algebra $U_q(\sgt)$ has no nontrivial finite dimensional simple module.
\end{theorem}

\section{Quantum toroidal algebras}


Let ${C}=(c_{i,j})_{0\leq i,j\leq n}$ be a symmetrizable  Cartan matrix. So there exists
a diagonal matrix $$D=diag(d_0,\cdots,d_n)$$ with $d_i\in\mz_+$ such that $\gcd(d_0,\cdots,d_n)=1$  and $DC$ is symmetric.

In this paper, we always assume that $(c_{i,j})_{1\leq i,j\leq n}$ is of finite type $X_n$ ($\not=A_1$) and
 $C$ is of affine type $X_n^{(1)}$. They are the Cartan matrices of Lie algebras $\sg$  and $\widehat{\sg}$, respectively.

For convenience, we use the following standard notations:
\beqs &&q_i=q^{d_i}, [m]_i=\frac{q_i^m-q_i^{-m}}{q_i-q_i^{-1}}, m\in\mz, \\
&&[m]_i!=\prod_{k=1}^m[k]_i, \left[{m\atop k}\right]_i=\frac{[m]_i!}{[k]_i![m-k]_i!}, m\geq 0.\eeqs

\subsection{quantum toroidal algebras}
\begin{defi} The quantum toroidal algebra $U_q(\sgt)$ is an associative $\mc(q)$-algebra generated by elements
$x_i^\pm(k), a_i(l), K_i^{\pm1},\gamma^{\pm\frac12}\, (i=0,1,\cdots,n,k\in\mz, l\in\mz\setminus\{0\}$ satisfying 
\beq
&&\gamma^{\pm\frac12} {\rm\; are\; central\; and\;} K_iK_j=K_jK_i,\; K_i^{\pm1}K_i^{\mp1}=1,\\
&&[a_i(l), K_i^{\pm1}]=0,\\
&&[a_i(l), a_j(l')]=\delta_{l+l',0}\frac{[lc_{i,j}]_i}{l}\cdot\frac{\gamma^l-\gamma^{-l}}{q-q^{-1}},\\
&&K_ix_j^\pm(k)K_i^{-1}=q^{\pm c_{i,j}}x_j^\pm(k),\\
&&[a_i(l), x_j^\pm(k)]=\pm\frac{[lc_{i,j}]_i}{l}\gamma^{\mp\frac|l|2}x_j^\pm(l+k),\\
&&[x_i^\pm(k+1), x_j^\pm(k')]_{q_i^{\pm c_{i,j}}}=-[x_j^\pm(k'+1), x_i^\pm(k)]_{q_i^{\pm c_{i,j}}},\\
&&[x_i^+(k), x_j^-(k')]=\delta_{i,j}\frac{\gamma^{\frac{k-k'}2}\phi^+_i(k+k')-\gamma^{-\frac{k-k'}2}\phi^-_i(k+k')}{q_i-q_i^{-1}},
\eeq
and the $q$-Serre relations
\beqs Sym_{t_1,\ldots,t_m}\sum_{k=0}^{m=1-c_{i,j}}(-1)^k\left[\begin{array}{c}1-c_{i,j}\\k\end{array}\right]_ix_i^\pm(t_1)\cdots x_i^\pm(t_k)x_j(t)x_i^\pm(t_{k+1})\cdots x_i^\pm(t_m)=0, \eeqs
where  $Sym_{t_1,\ldots,t_m}$ denotes the symmetrization with respect to the indices $(t_1,\cdots,t_m)$
and \beqs&& \phi^\pm_i(z)=\sum_{m=0}^\infty\phi^\pm_i(\pm m)z^{\mp m}=K_i^{\pm1}\exp\left(\pm(q_i-q_i^{-1})\sum_{l=1}^\infty a_i(\pm l)z^{\mp l}\right).\eeqs
\end{defi}

\subsection{quantum affine algebras}

\begin{defi} The horizontal quantum affine algebra $U_q(\sga)$ is the subalgebra of $U_q(\sgt)$ generated by  elements
\beqs \{x_i^\pm(0), K_i^{\pm1}, (i=0,1,\cdots,n)\}.\eeqs
\end{defi}

Let $\alpha_i$ be the simple root associated to $x_i^+(0)$ and $\delta=\sum_{i=0}^ns_i\alpha_i$ the primitive imaginary root. Then
$\Gamma=\prod_{i=0}^nK_i^{s_i}$ is a central element.

\def\cu{\mathcal{U}}

\begin{defi} The vertical quantum affine algebra $\cu_q(\sga)$ is the subalgebra of $U_q(\sgt)$ generated by  elements
\beqs x_i^\pm(k), a_i(l), K_i^{\pm1},\gamma^{\pm\frac12}\, (i=1,\cdots,n,k\in\mz, l\in\mz\setminus\{0\}.\eeqs
\end{defi}

Adding $\Gamma^\frac12$ to $U_q(\sga)$,  it is well known that the extended algebra is isomorphic to  $\cu_q(\sga)$. In particular, there exists an isomorphism $\Theta$ such that
\beqs &&\Theta(x_i^\pm(0))=x_i^\pm(0),\,i=1,\cdots,n,\\
&&\Theta(\Gamma^\frac12)=\gamma^\frac12.\eeqs

\section{Highest weight modules over quantum affine algebras}

In this section, we introduce the notion {\bf highest weight module of Kac-Moody type}.

For arbitrary given $m_0,\cdots,m_n\in\mz$, let $A(m_0,\cdots, m_n)$  denote the subalgebra generated by $\{x_i^\pm(\pm m_i), (\gamma^{m_i}K_i)^{\pm1}, (i=0,1,\cdots,n)\}$.
\begin{lemma}
$A(m_0,\cdots, m_n)$ is isomorphic to $U_q(\sga)$.
\end{lemma}
\begin{proof}
In fact, the map
\beqs \pi: x_i^\pm(0)\mapsto x_i^\pm(\pm m_i), K_i^{\pm1}\mapsto (\gamma^{m_i}K_i)^{\pm1}\eeqs
defines an isomorphism of algebras. In the following we shall identify $A(m_0,\cdots, m_n)$ as $U_q(\sga)$ by this isomorphism.
\end{proof}

Let $N^\pm$ be the subalgebra generated by $ \{x_i^\pm(0), (i=0,1,\cdots,n)\}$ and let $N^0$ be the Laurent polynomial algebra $\mc(q)[K_0^{\pm1},\cdots,K_n^{\pm1}]$, then
\beqs U_q(\sga)&=&N^-\cdot N^0\cdot N^+.\eeqs

\begin{defi}
(a) Suppose that $V$ is a $U_q(\sga)$-module and $v\in V$. If $x_i^+(0)v=0, K_i^{\pm}v=a_i^{\pm1}v$ for all $i$, then $v$ is called a highest weight vector of Kac-Moody type.

(b) A module generated by a highest weight vector of Kac-Moody type is called a highest weight module of Kac-Moody type.

(c) If $v$ is a highest weight vector of Kac-Moody type, then $\mc v$ is a one dimensional module over $N^0\cdot N^+$. The induced module $M(v)=U_q(\sga)\otimes_{N^0\cdot N^+}\mc v$ is called a Verma module of  Kac-Moody type.
\end{defi}

\begin{lemma}
Any highest weight module of Kac-Moody type is a quotient of some Verma module of Kac-Moody type.
\end{lemma}
\begin{proof}It follows by the definition.
\end{proof}

\begin{lemma}
Assume that $v$ is a highest weight vector  of Kac-Moody type and $V$ is a simple $U_q(\sga)$-module  generated by $v$. If $\dim V<\infty$, then $V=\mc v$ is trivial.
\end{lemma}
\begin{proof}
First we claim that any simple finite dimensional $U_q(\sga)$-module $V$ is a simple $\cu_q(\sga)$-module.

Note that $\gamma$ is central and it acts as a scalar over $V$. By relation $[a_1(1), a_1(-1)]=[2]_1\frac{\gamma-\gamma^{-1}}{q-q^{-1}}$, if $\gamma\not=\pm1$,  the subalgebra generated by  $a_1(1)|_V, a_1(-1)|_V$ is isomorphic to the Weyl algebra $\mathfrak{A}$, which has no finite dimensional module (\cite{B}). So $\Gamma|_V=\gamma|_V=\pm1$.

For each $i$, $v$ is a highest weight vector of the quantum group of type $A_1$ generated by $x_i^\pm(0), K_i^{\pm1}$. So $\dim V<\infty$ implies $K_iv=\varepsilon_i q^{\lambda_i}v$ such that $\lambda_i\in\mathbb{N}, \varepsilon_i\in\{1,-1\}$ (see Theorem 2.6 of \cite{Jan}).

Moreover, we infer that $q^{s_0\lambda_0+\cdots+s_n\lambda_n}\in\{1,-1\}$. Since $q$ is generic, we have $s_0\lambda_0+\cdots+s_n\lambda_n=0$ and  $\lambda_0=\cdots=\lambda_n=0$. Then the Verma module $M(v)$ has a unique maximal submodule $J$ generated by $\{x_i^-(0)v|i=0,1,\cdots\}$. So $V=M(v)/J\cong \mc(q) v$.
\end{proof}

\section{Proof for main theorem}

 Throughout this section, we always assume that $q$ is  generic and $V$ is a finite-dimensional simple $U_q(\sgt)$-module. So $\gamma|V=\pm1$. For convenience, we assume $\gamma|V=1$. The proof for $\gamma|V=-1$ is very similar.

Define matrix $C(q,l)=([lc_{i,j}]_i)_{0\leq i,j\leq n}$ for all $l\not=0$. Then there exits a polynomial $\det_\sg$ such that
\beqs \det(C(q,l))&=&\frac{\det_\sg(q^l)}{\prod_{i=0}^n(q_i-q_i^{-1})}.\eeqs
The polynomial $\det_\sg$ is explicitly give by the following tabular:
\vskip3mm
\centerline{ }
\begin{tabular}{|c|c|}
\hline Type  of $\sg$& $\det_\sg(q)$\\
\hline
${A_n} (n\geq2)$&$(1-q^{-2})^{n+1}(q^{n+1}-1)^2$\\
${B_n}(n\geq2)$&$(q^2-1)(q^4-1)^{n}q^{-6}(q^4-q^{-4})$\\
${C_n} (n\geq3)$&$(q-q^{-1})^{n+1}(q+q^{-1})^3q^{n-3}$\\
${D_n} (n\geq4)$&$(q-q^{-1})^{n+2}(q+q^{-1})^2(q+1)(q^{n-3}-q^{-n+2})$\\
${E_6} $&$(q-q^{-1})^6(q^2-q^{-2})(q^{3}-q^{-3})^2$\\
${E_7} $&$(q-q^{-1})^7(q^2-q^{-2})(q^{3}-q^{-3})(q^4-q^{-4})$\\
${E_8} $&$(q-q^{-1})^8(q^2-q^{-2})^3(q^{3}-q^{-3})(q^5-q^{-5})$\\
${F_4} $&$(q-q^{-1})(q^2-q^{-2})^3(q^{3}-q^{-3})(q^4-q^{-4})(q^5-q^{-5})$\\
${G_2} $&$q^{-10}(q^2-1)(q^4-1)(q^6-1)^3$\\
\hline\end{tabular}

\subsection{some useful lemmas}

\begin{lemma}
$C(q,l)$ is invertible for all $l\not=0$.\end{lemma}
\begin{proof}
Straightforward.
\end{proof}

\begin{lemma}
There exist elements $\xi(l)$ for all $l\in\mz\setminus\{0\}$ such that
\beq [\xi(l), x_j^\pm(k)]&=&\pm x_j^\pm(l+k), \quad\forall k\in\mz, j=0,1,\cdots,n.\eeq
\end{lemma}
\begin{proof}
Because $C(q,l)$ is invertible, let $(\xi_0,\cdots,\xi_n)$ be the unique solution of
\beqs (\xi_0,\cdots,\xi_n)C(q,l)=(l,\cdots,l),\eeqs
and let $\xi(l)=\sum_{i=0}^n\xi_ia_i(l)$. Then
\beqs [\xi(l), x_j^\pm(k)]=\pm\sum_{i=0}^n\xi_i\frac{[lC_{i,j}]_i}{l}x_j^\pm(l+k)=\pm x_j^\pm(l+k).\eeqs
\end{proof}

For convenience, we write
\beq x_i^\pm\otimes f(z)&:=&\sum f_jx_i^\pm(j)\eeq
for all $0\leq i\leq n$ and $f(z)=\sum f_j z^j\in\mc[z, z^{-1}]$.

\begin{lemma}
There exists $f(z)=\sum_{j=0}^nf_jz^j\in\mc[z]$ such that $f_0=-1$ and $x_i^\pm\otimes f(z)V=0$ for each $i$.
\end{lemma}
\begin{proof}
Since $\dim V<\infty$, there are polynomials $g^\pm_i(z)\in\mc[z, z^{-1}]$ such that
\beqs x_i^\pm\otimes g_i^\pm(z)V=0.\eeqs

If $p(z)=g_i^\pm(z)h(z)$, then $x_i^\pm\otimes p(z)=\pm[\sum h_l\xi(l), x_i^\pm\otimes g_i^\pm(z)]$,
then $x_i^\pm\otimes p(z)V=0$.
Let $f(z)=cz^mg_0^-(z)\cdots g_n^-(z)g_0^+(z)\cdots g_n^+(z)$, where $c\in \mc^*$ and $m\in\mz$ such that $f(z)\in\mc[z]$ and $f(0)=-1$.
\end{proof}

Let $\Xi=\sum_{l=1}^{\deg(f)}f_l\xi(l)$. Assume $v\in V$ is an eigenvector of $\Xi$ with eigenvalue $\lambda$.

\begin{lemma}
For all $0\leq i_1,\cdots, i_k\leq n$ and $m_1, \cdots, m_k\in\mz$, we have
\beq \Xi\cdot(x^\pm_{i_1}(m_1)\cdots x^\pm_{i_k}(m_k)v)=(\lambda\pm k)x^\pm_{i_1}(m_1)\cdots x^\pm_{i_k}(m_k)v.\eeq
\end{lemma}
\begin{proof}
It follows from  $x_i^\pm\otimes z^mf(z)V=0$ and
\beqs [\Xi|_V, x_i^\pm(m)|_V]=\pm(x_i^\pm\otimes z^mf(z)+x_i^\pm(m))|_V=\pm x_i^\pm(m)|_V.\eeqs
\end{proof}

\subsection{proof of main theorem}

Because $V$ is finite-dimensional, by (4.3), there exists $v'$ of $U_q(\sga)$ such that $\Xi\cdot v'=\lambda'v'$
and
\beq x_i^+(m)v'&=&0, \quad\forall m\in\mz, i=0,\cdots,n.\eeq


By (4.4), $U_q(\sga)v'$ is a finite dimensional highest weight module of $A(m_0,m_1,\cdots,m_n)$. By Lemma 3.4, $\dim A(m_0,m_1,\cdots,m_n)v'<\infty$ and Theorem 2.6 of \cite{Jan},  we have $K_iv'=\epsilon_iv',\epsilon_i\in\{1,-1\}$ for all $i=0,\cdots,n$.

If there exists $v'':=x_i^-(m_i)v'\not=0$ for some index $i$ and integer $m_i$, then $v''$ is also a highest weight vector of $U_q(\sga)$ and $K_iv''=\pm q_i^{-2}v''\not=\pm v''$, this forces $\dim V=\infty$, a contradiction. So we also have $x_i^-(m)v'=0$ for all $m\in\mz$ and $i=0,1,\cdots,n$.  This by (2.7) also implies
\beqs a_i(l)v'=0, l\in\mz\setminus\{0\}, i=0,\cdots,n.\eeqs
So $V=\mc v'$ is  trivial.

\vskip30pt \centerline{\bf ACKNOWLEDGMENTS}

\vskip10pt The author gratefully acknowledges the partial financial support from the NNSF (Nos. 11871249, 11771142) and the Jiangsu Natural Science Foundation (No. BK20171294).  Part of this work was done during the author's visiting Paris Diderot. The author would like to thank Prof. Marc Rosso for warm hospitality and helpful discussions.


\begin{thebibliography}{ABC}

\bibitem[1]{B} R. E. Block, {\sl The irreducible representations of the Lie algebra $\mathfrak{sl}(2)$ and of the Weyl algebra}. Adv. Math. 39 (1981), 69-110.

\bibitem[2]{CP1} V. ChariV, A. Pressley, \textit{Quantum affine algebras}, Comm. Math. Phys. \textbf{142} (1991), 261--283.
\bibitem[3]{CP2} V. ChariV, A. Pressley, \textit{minimal affinizations of representations of quantum groups: The nonsimply-laced case}, Lett. Math. Phys.  \textbf{35} (1995), 99--114.
\bibitem[4]{CP3} V. ChariV, A. Pressley, \textit{minimal affinizations of representations of quantum groups: The simply-laced case}, J. Algebra. \textbf{184} (1996), 1--30.
\bibitem[5]{GJ} Y.~Gao and N. Jing,
\textit{$U_q(\mathfrak{gl}^N)$ action on $\mathfrak{gl}^N$-modules and quantum toroidal algebras}, J. Algebra \textbf{273} (2004), no. 1, 320--343.

\bibitem[6]{GKV} V.~Ginzburg, M.~Kapranov and E.~Vasserot, \textit{Langlands reciprocty for algebric surfaces,}  Math. Res. Lett. \textbf{2} (1995), 147--160.

\bibitem[7]{FJW} I.~B.~Frenkel, N.~Jing and W.~Wang,  \textit{Quantum vertex representations via finite groups and the McKay correspondence}, Comm. Math. Phys. \textbf{211} (2000), 365--393.
\bibitem[8]{H} D.~Hernandez, \textit{Quantum toroidal algebras and their representations},
Selecta Math. (N.S.) \textbf{14} (2009), 701--725.

\bibitem[9]{Jan} J. C. Jantzen, \textit{Lectures on Quantum groups}, A.M.S. Providence, (1996).

\bibitem[10]{M1} K. Miki, \textit{ Toroidal and level 0 $U_q'(\widehat{sl}_{n+1})$ actions on $U_q(\widehat{gl}_{n+1})$-modules} J. Math. Phys. \textbf{40} (1999), 3191--3210.

\bibitem[11]{M2} K. Miki, \textit{Representations of quantum toroidal algebra $U_q(sl_{n+1},tor)(n>2)$}, J. Math. Phys. \textbf{41} (2000), 7079-7098.

\bibitem[12]{M3} K. Miki, \textit{Quantum toroidal algebra $U_q(sl_2,tor)$ and R matrices}, J. Math. Phys. \textbf{42} (2001), 2293-2308.


\bibitem[13]{Sa} Y. ~Saito, \textit{Quantum toroidal algebras and their vertex representations}, Publ. RIMS. Kyoto Univ.,  \textbf{34} (1998), 155--177.

\bibitem[14]{VV} M. Varagnolo and E. Vasserot, \textit{Schur duality in the toroidal setting},  Comm. Math.
Phys. \textbf{182} (1996), 469--484.


\end{thebibliography}
\end{document}